\documentclass[twoside, a4paper, 10pt]{amsart}
\title[ ]{Glasner property for unipotently generated group actions on tori}
\usepackage{amsfonts}
\usepackage{amsthm}
\usepackage{verbatim}
\usepackage{amsmath, amssymb}
\usepackage{tikz}
\usetikzlibrary{matrix, arrows}
\usepackage{listings}
\usepackage{color}
\usepackage{listings}
\usepackage[all]{xy}
\usepackage[pdftex,colorlinks,linkcolor=blue,citecolor=blue]{hyperref}
\usepackage{graphicx}
\usepackage{float}
\usepackage[margin=3cm]{geometry}
\usepackage{bigints}
\usepackage{dsfont}
\author{Kamil Bulinski}
\address{School of Mathematics and Statistics, University of Sydney, Australia}
\email{kamil.bulinski@sydney.edu.au}

\author{Alexander Fish}
\address{School of Mathematics and Statistics, University of Sydney, Australia}
\email{alexander.fish@sydney.edu.au}

\begin{document}
\maketitle
\raggedbottom

\newcommand{\cA}{\mathcal{A}}
\newcommand{\cB}{\mathcal{B}}
\newcommand{\cC}{\mathcal{C}}
\newcommand{\cD}{\mathcal{D}}
\newcommand{\cE}{\mathcal{E}}
\newcommand{\cF}{\mathcal{F}}
\newcommand{\cG}{\mathcal{G}}
\newcommand{\cH}{\mathcal{H}}
\newcommand{\cI}{\mathcal{I}}
\newcommand{\cJ}{\mathcal{J}}
\newcommand{\cK}{\mathcal{K}}
\newcommand{\cL}{\mathcal{L}}
\newcommand{\cM}{\mathcal{M}}
\newcommand{\cN}{\mathcal{N}}
\newcommand{\cO}{\mathcal{O}}
\newcommand{\cP}{\mathcal{P}}
\newcommand{\cQ}{\mathcal{Q}}
\newcommand{\cR}{\mathcal{R}}
\newcommand{\cS}{\mathcal{S}}
\newcommand{\cT}{\mathcal{T}}
\newcommand{\cU}{\mathcal{U}}
\newcommand{\cV}{\mathcal{V}}
\newcommand{\cW}{\mathcal{W}}
\newcommand{\cX}{\mathcal{X}}
\newcommand{\cY}{\mathcal{Y}}
\newcommand{\cZ}{\mathcal{Z}}
\newcommand{\bA}{\mathbb{A}}
\newcommand{\bB}{\mathbb{B}}
\newcommand{\bC}{\mathbb{C}}
\newcommand{\bD}{\mathbb{D}}
\newcommand{\bE}{\mathbb{E}}
\newcommand{\bF}{\mathbb{F}}
\newcommand{\bG}{\mathbb{G}}
\newcommand{\bH}{\mathbb{H}}
\newcommand{\bI}{\mathbb{I}}
\newcommand{\bJ}{\mathbb{J}}
\newcommand{\bK}{\mathbb{K}}
\newcommand{\bL}{\mathbb{L}}
\newcommand{\bM}{\mathbb{M}}
\newcommand{\bN}{\mathbb{N}}
\newcommand{\bO}{\mathbb{O}}
\newcommand{\bP}{\mathbb{P}}
\newcommand{\bQ}{\mathbb{Q}}
\newcommand{\bR}{\mathbb{R}}
\newcommand{\bS}{\mathbb{S}}
\newcommand{\bT}{\mathbb{T}}
\newcommand{\bU}{\mathbb{U}}
\newcommand{\bV}{\mathbb{V}}
\newcommand{\bW}{\mathbb{W}}
\newcommand{\bX}{\mathbb{X}}
\newcommand{\bY}{\mathbb{Y}}
\newcommand{\bZ}{\mathbb{Z}}

\newcounter{dummy} \numberwithin{dummy}{section}

\theoremstyle{definition}
\newtheorem{mydef}[dummy]{Definition}
\newtheorem{prop}[dummy]{Proposition}
\newtheorem{corol}[dummy]{Corollary}
\newtheorem{thm}[dummy]{Theorem}
\newtheorem{lemma}[dummy]{Lemma}
\newtheorem{eg}[dummy]{Example}
\newtheorem{notation}[dummy]{Notation}
\newtheorem{remark}[dummy]{Remark}
\newtheorem{claim}[dummy]{Claim}
\newtheorem{Exercise}[dummy]{Exercise}
\newtheorem{question}[dummy]{Question}

\begin{abstract}
A theorem of Glasner from 1979 shows that if $A \subset \bT = \bR/\bZ$ is infinite then for each $\epsilon > 0$ there exists an integer $n$ such that $nA$ is $\epsilon$-dense and Berend-Peres later showed that in fact one can take $n$ to be of the form $f(m)$ for any non-constant $f(x) \in \bZ[x]$. Alon and Peres provided a general framework for this problem that has been used by Kelly-L\^{e} and Dong to show that the same property holds for various linear actions on $\bT^d$. We complement the result of Kelly-L\^{e} on the $\epsilon$-dense images of integer polynomial matrices in some subtorus of $\bT^d$ by classifying those integer polynomial matrices that have the Glasner property in the full torus $\bT^d$. We also extend a recent result of Dong by showing that if $\Gamma \leq \operatorname{SL}_d(\bZ)$ is generated by finitely many unipotents and acts irreducibly on $\bR^d$ then the action $\Gamma \curvearrowright \bT^d$ has a uniform Glasner property.
\end{abstract}

\section{Introduction}

In 1979 Glasner \cite{Glasner79} showed that an infinite subset $A \subset \bT = \bR/\bZ$ satisfies the property that for every $\epsilon>0$ there exists $n \in \bN$ such that $nA$ is $\epsilon$-dense in $\bT$. This was later extended by Berend-Peres \cite{Berend-Peres} in a number of ways. For example, they showed that for each non-constant polynomial $f(x) \in \bZ[x]$ there exists $n \in \bN$ such that $f(n)A$ is $\epsilon$-dense in $\bT$. This motivated them to define a set $S \subset \bN$ to be \textit{Glasner} if for all infinite $A \subset \bT$ and $\epsilon>0$ there exists an $s \in S$ such that $sA$ is $\epsilon$-dense. Turning our attention to more general semigroup actions on metric spaces, we extend this definition as follows.

\begin{mydef} We say that a subset $S$ of a semigroup $\Gamma$ is \textit{Glasner for an action} $\Gamma \curvearrowright X$ on a compact metric space $X$ by continuous maps if for each infinite $Y \subset X$ and $\epsilon >0$ there exists an $s \in S$ such that $sY$ is $\epsilon$-dense. We say that the action $\Gamma \curvearrowright X$ is Glasner if $\Gamma$ is a Glasner set with respect to this action. \end{mydef}

In fact, Berend-Peres realised that a more uniform notion of the Glasner property holds for this action on $\bT$. This leads us to the following definition.

\begin{mydef} If $k: \bR_{>0} \to \bN$ is a function then we say that a subset $S$ of a semigroup $\Gamma$ is $k$-\textit{uniformly Glasner for an action} $\Gamma \curvearrowright X$ on a compact metric space $X$ by continuous maps if there is an $\epsilon_0 >0$ such that for each $0 < \epsilon < \epsilon_0$ and $Y \subset X$ with $|Y|\geq k(\epsilon)$ there exists an $s \in S$ such that $sY$ is $\epsilon$-dense. We say that the action $\Gamma \curvearrowright X$ is $k$-uniformly Glasner if $\Gamma$ is a $k$-uniformly Glasner set with respect to this action. We will also use the phrase \textit{uniformly Glasner} to mean $k$-\textit{uniformly Glasner} for some unspecified $k: \bR_{>0} \to \bN$. \end{mydef}

In particular, Berend-Peres showed that the multiplicative action of $\bN$ acting on $\bT$ is $\left(c_1/\epsilon\right)^{c_2/\epsilon}$-uniformly Glasner\footnote{If $k(\epsilon,c, c_1, \ldots,)$ is an expression involving $\epsilon$ and possibly constants $c, c_i$ etc., by $k(\epsilon)$-uniformly Glasner we always technically mean $k$-uniformly Glasner for the function $k(\epsilon) = k(\epsilon,c ,c_1,\ldots)$ for some choice of $c, c_i>0$.}. Moreover, they also gave a lower bound by showing that there is a set $A_{\epsilon} \subset \bT$ of cardinality $c\epsilon^{-2}$ such that $nA_{\epsilon}$ is not $\epsilon$-dense for all $n \in \bN$. The seminal work of Alon-Peres \cite{Alon-Peres} closed this significant difference in the lower and upper bounds by showing that in fact this action is $\epsilon^{-2-\delta}$-uniformly Glasner for all $\delta>0$. Secondly, Alon-Peres also quantitatively improved the polynomial example by showing that if $f(x) \in \bZ[x]$ is a non-constant polynomial of degree $D$, then the set $\{f(n) ~|~ n \in \bN \}$ is $\epsilon^{-2D - \delta}$-uniformly Glasner for all $\delta>0$.

The Glasner property of linear actions on a higher-dimensional torus $\bT^d$ was studied by Kelly- L\^{e} \cite{Kelly-Le}, where they used the techniques of Alon-Peres \cite{Alon-Peres} to show that the natural action of the multiplicative semi-group $M_{d \times d}(\bZ)$ of $d \times d$ integer matrices on $\bT^d$ is $c_d\epsilon^{-3d^2}$uniformly Glasner. This was later improved by Dong in \cite{Dong1} where he showed, using the same techniques of Alon-Peres together with the deep work of Benoist-Quint \cite{Benoist-Quint}, that the action $\operatorname{SL}_d(\bZ) \curvearrowright \bT^d$ is $c_{\delta, d} \epsilon^{-4d - \delta}$-uniformly Glasner for all $\delta >0$. Furthermore, Kelly-L\^{e}, also gave the following multidimensional generalization of the aforementioned result on the Glasner property of polynomial sequences.

\begin{thm}[Kelly-L\^{e}, Theorem 2 in \cite{Kelly-Le}] \label{thm: Kelly-Le polynomial} Let $A(x) \in M_{d \times d}(\bZ[x])$ be a matrix with integer polynomial entries. Then the following conditions are equivalent.

\begin{enumerate}
	\item The columns of $A(x) - A(0)$ are linearly independent over $\bZ$ (as elements in $\bZ[x]^d$) and whenever $v, w \in \bZ^d$ are such that $$ v \cdot (A(x) - A(0)) w = 0$$ then $v \cdot A(0)w= 0.$

	\item For any infinite subset $Y \subset \bT^d$ there exists a subtorus (non-trivial connected closed Lie subgroup) $\mathcal{T} = \mathcal{T}(Y, A(x))$ such that for all $\epsilon >0$ there exists an $n \in \bZ$ such that, for some $Y_0 \subset Y$, the set $A(n)Y_0 = \{A(n)y ~|~ y \in Y_0 \}$ is $\epsilon$-dense in a translate of $\mathcal{T}$.

\end{enumerate}

\end{thm}

The following main result of this paper characterizes those $A(x) \in M_{d \times d}(\bZ[x])$ which satisfy the stronger property that $\{A(n) ~|~ n \in \bZ\}$ is Glasner (for the natural linear action on $\bT^d$), i.e., it characterizes when we can take the subtorus $\mathcal{T}$ to be the full $\bT^d$. 

\begin{thm} \label{thm: main polynomial intro} Let $A(x) \in M_{d \times d}(\bZ[x])$ be a matrix with integer polynomial entries. Then the following conditions are equivalent.

\begin{enumerate}
	\item \label{hyperplane fleeing main thm} For all $v \in \bZ^d\setminus\{ 0 \}$ and $w \in \bZ^d\setminus\{0\}$ we have that $$v \cdot (A(x) - A(0))w \neq 0.$$ 
	\item \label{glasner conditon main thm} The set $\{A(n) ~|~ n \in \bZ \}$ is $c_1 \epsilon^{-c_2}$-uniformly Glasner for the linear action $M_{d \times d}(\bZ[x]) \curvearrowright \bT^d$ for some constants $c_1, c_2  >0$ depending on $A(x)$. That is, for every $Y \subset \bT^d$ with $|Y| > c_1 \epsilon^{-c2}$ there exists $n \in \bZ$ such that $A(n)Y$ is $\epsilon$-dense in $\bT^d$.
\end{enumerate}

\end{thm}

\begin{remark} As we shall see in the $(\ref{glasner conditon main thm}) \implies (\ref{hyperplane fleeing main thm})$ proof, in condition (\ref{glasner conditon main thm}) of Theorem~\ref{thm: main polynomial intro} one can replace \textit{$c_1 \epsilon^{-c_2}$-uniformly Glasner} with the weaker condition of being just Glasner. So Glasner and $c_1 \epsilon^{-c_2}$-uniformly Glasner are equivalent for sets of the form $\{A(n) ~|~ n \in \bZ \}$ for some $A(x) \in M_{d \times d}(\bZ[x])$.

\end{remark}

Let us remark that, as stated, the subtorus $\mathcal{T}$ in Theorem~\ref{thm: Kelly-Le polynomial} depends on $Y$ and not just $A(x)$ and the proof in \cite{Kelly-Le} is not constructive as it makes use of Ramsey's Theorem on graph colourings to demonstrate the existence of such a $\mathcal{T}$. Thus it does not seem that our result can be easily derived from the result or techniques of Kelly-L\^{e}. Note that in Theorem~\ref{thm: quantitative Glasner polynomial} we will provide an effective estimate on the uniformity (estimates on the constants $c_1$ and $c_2$).

It will be convenient to give some alternative formulations and geometrically intuitive extensions of condition (\ref{hyperplane fleeing main thm}) in Theorem~\ref{thm: main polynomial intro}.

\begin{mydef} A set $S \subset \bR^d$ is said to be \textit{hyperplane-fleeing} if for all proper affine subspaces $H$ of $\bR^d$ (i.e., $H = W + a$ for some proper vector subspace $W \subset \bR^d$ and $a \in \bR^d$) we have that $S \not\subset H$. \end{mydef}

Thus, condition (\ref{hyperplane fleeing main thm}) in Theorem~\ref{thm: main polynomial intro} is equivalent to the statement that for each non-zero $w \in \bZ^d \setminus \{0 \}$ the orbit $\{A(n)w ~|~ n \in \bZ\}$ is hyperplane-fleeing (as it is not a subset of the hyperplane $\{x \in \bR^d ~|~ v \cdot x - v \cdot A(0)w = 0\}$ for any $v \in \bZ^d \setminus \{0\}$ and in fact any $v \in \bR^d \setminus \{ 0\}$ as $A(x)$ has integer polynomial entries). This hyperplane-fleeing property of the orbits is related to the irreducibility of linear group actions. Indeed, it is easy to see that if $d>1$, and $\Gamma \leq M_{d \times d}(\bZ)$ is a semigroup whose action on $\bR^d$ is irreducible, then the orbit of any non-zero vector $v \in \bR^d \setminus \{0 \}$ is hyperplane-fleeing (we prove a stronger statement in Lemma~\ref{lemma: hyperplane fleeing Cayley ball}). This enables us to use Theorem~\ref{thm: main polynomial intro} to deduce the Glasner property for various irreducible representations. For instance, we recover in a more elementary way (by avoiding the deep work of Benoist-Quint \cite{Benoist-Quint}) the aforemnentioned result of Dong but with weaker (but still polynomial in $\epsilon^{-1}$) uniformity bounds. In general, we will demonstrate that subgroups generated by a finite set of unipotent elements of $\operatorname{SL}_d(\bZ)$ that act irreducibly on $\bT^d$ satisfy the uniform Glasner property.

\begin{thm}\label{thm: main thm unipotent intro} Let $d>1$ and let $u_1, \ldots, u_m \in \operatorname{SL}_d(\bZ)$ be unipotent elements such that the action of the subgroup $\Gamma = \langle u_1, \ldots, u_m \rangle$ on $\bR^d$ is irreducible. Then there exists $c_1, c_2$ (depending on $\Gamma$) such that the following is true: For each $\epsilon > 0$ there exists an integer $k \leq c_1 \epsilon^{-c_2}$ such that for any distinct $x_1, \ldots, x_k \in \bT^d$ there exists $\gamma \in \Gamma$ such that $\{\gamma x_1, \ldots, \gamma x_k \}$ is $\epsilon$-dense in $\bT^d$. In other words, the action of $\Gamma$ on $\bT^d$ is $c_1\epsilon^{-c_2}$ uniformly Glasner.

\end{thm}

This will follow by showing (see Proposition~\ref{prop: unipotent implies polynomial}) that $\Gamma$ contains such a polynomial satisfying the condition (\ref{hyperplane fleeing main thm}) of Theorem~\ref{thm: main polynomial intro}.

Let us now explore some examples of such subgroups other than $\operatorname{SL}_d(\bZ)$ (which is an example as $\operatorname{SL}_d(\bZ)$ is generated by the finitely many elementary matrices obtained from changing a single $0$ to a $1$ in the identity matrix).

\begin{thm} Let $Q(x,y,z) = xy - z^2$ or $Q(x,y,z) = x^2 - y^2 - z^2$. Let $\Gamma = \operatorname{SO}_{\bZ}(Q)$ be subgroup of $\operatorname{SL}_d({\bZ})$ preserving this quadratic form. Then the action of $\Gamma$ on $\bT^d$ is uniformly Glasner.

\end{thm}

\begin{proof}
For $Q(x,y,z) = xy - z^2$ this can be seen as follows. By identifying $(x,y,z) \in \bZ^3$ with $$ \begin{bmatrix}
    z   & -y \\
    x  & -z \\
\end{bmatrix} \in \mathfrak{sl}_2(\bZ)$$ we see that $Q(x,y,z)$ is the determinant. But the determinant is preserved by the conjugation action (adjoint representation) of $\operatorname{SL}_2(\bZ)$ on $\mathfrak{sl}_2(\bR)$ given by $$\operatorname{Ad}(g)A = gAg^{-1} \quad \text{for } A \in \mathfrak{sl}_2(\bR) \text{ and } g \in \operatorname{SL}_2(\bZ),$$ which is irreducible. Note that $\operatorname{Ad}(u)$ is unipotent for unipotent $u$ since $u$ is a polynomial map and group homomorphism, $\operatorname{Ad}(\operatorname{SL}_2(\bZ))$ thus generated by unipotents. Of course, this example generalizes to any higher dimensional adjoint representation, thus showing that it also has the uniform Glasner property. For $Q(x,y,z) = x^2 - y^2 - z^2$ one instead notices

 $$Q(x,y,z)= \operatorname{det} \begin{pmatrix} z & -(x+y) \\ x-y & -z \end{pmatrix}.$$ Hence we may regard $Q$ as the determinant map on the abelian subgroup $$ \left \{ \begin{pmatrix} a_{11} & a_{12} \\ a_{21} & a_{22} \end{pmatrix} \in \mathfrak{sl}_2(\bZ) \text{ }| \text{ } a_{21} \equiv a_{12} \mod 2 \right\} \cong \bZ^3.$$ Now notice that the conjugation action of  $$\left \langle \begin{pmatrix}1 & 2 \\ 0 & 1 \end{pmatrix}, \begin{pmatrix}1 & 0 \\ 2 & 1 \end{pmatrix} \right \rangle,$$ preserves this additive subgroup and acts irreducibly on $\mathfrak{sl}_2(\bR)$. Again, the generators are unipotent hence have unipotent image under the adjoint representation, as required.
\end{proof}

We remark that these examples complement a recent work of Dong \cite{Dong2} where he extended his result from \cite{Dong1} on the Glasner property of $\operatorname{SL}_d(\bZ) \curvearrowright \bT^d$ by showing that the subgroups $\Gamma \leq \operatorname{SL}_d(\bZ)$ that are Zariski dense in $\operatorname{SL}_d(\bR)$ are also Glasner for the action on $\bT^d$, but the uniform Glasner property was not established. The examples above are not Zariski dense in $\operatorname{SL}_d(\bR)$, though it is remarked in Remark 4.2 of \cite{Dong2} that it is possible to also extend his techniques to the case where $\Gamma$ satisfies the Benoist-Quint hypothesis, which these examples do. However, these techniques are not quantitative and do not establish the uniform Glasner property provided in Theorem~\ref{thm: main thm unipotent intro}. It is also worth remarking that our proofs are more self-contained as they avoid the deep work of Benoist-Quint.

\textbf{Acknowledgement:} The authors were partially supported by by the Australian Research Council grant DP210100162.

\section{Hyperplane fleeing orbits implies Glasner property}

In this section we prove Theorem~\ref{thm: main polynomial intro}. We start with the easier direction.

\begin{proof}[Proof of (\ref{glasner conditon main thm}) $\implies$ (\ref{hyperplane fleeing main thm}) in Theorem~\ref{thm: main polynomial intro}] Suppose that we have $v, w \in \bZ^d \setminus\{0\}$ such that $v \cdot (A(x))w = c$ where $c = v \cdot A(0)w$ is a constant. Let $w_m \in \bT^d$ be the image of $\frac{1}{m}w$ and $c_m \in \bT$ be the image of $\frac{1}{m}c$. Notice that $C = \{c_m ~|~ m \in \bZ_{>0} \}$ cannot be dense in $\bT$ because $c_m \to 0 \in \bT$, hence avoids a non-empty open set $U \subset \bT$. The map $f:\bT^d \to \bT$ given by $f(u) = v \cdot u$ is well defined, continuous and surjective with $f(A(n)w_m) = c_m$ for all $n \in \bZ$ and $m \in \bZ_{>0}$. Thus the infinite set $Y = \{w_m ~|~ m \in \bZ_{>0}\} \subset \bT^d$ satisfies the property that $f(A(n)Y)$ will never intersect $U$ and so $A(n)Y$ will never intersect the non-empty open set $f^{-1}(U)$.
\end{proof}

\begin{lemma} \label{lemma: basic hyperplane-fleeing} If $A(x) \in M_{d \times d}(\bZ[x])$  is a matrix of integer polynomials then the following are equivalent.

\begin{enumerate} 

\item The orbit $\{A(n)w ~|~ n \in \bZ\}$ is hyperplane fleeing for all  $w \in \bZ^d\setminus\{0\}$.
\item For all $w \in \bZ^d\setminus\{0\}$, the entries of $(A(x) - A(0))w$ are polynomials in $\bZ[x]$ that are linearly independent over $\bZ$.
\item The polynomial $v^t (A(x) - A(0))w$ is non-zero for all $v,w \in \bZ^d\setminus\{0\}$.
\item For all $v \in \bZ^d\setminus\{0\}$, the entries of $v^t (A(x) - A(0))$ are polynomials in $\bZ[x]$ that are linearly independent over $\bZ$. 
\end{enumerate}

\end{lemma}

If $w = (w_1, \ldots, w_d) \in \bZ^d$ we let $\text{gcd}(w) = \text{gcd}(w_1, \ldots, w_d) $. If $\vec{w}_1, \ldots, \vec{w}_d$ are integer vectors (of possibly different dimensions) then we identify $(\vec{w_1}, \ldots, \vec{w_d})$ with their concatenation, so $\text{gcd}(\vec{w}_1, \ldots, \vec{w}_d)$ makes sense and is equal to $\text{gcd}(\text{gcd}(\vec{w}_1), \ldots, \text{gcd}(\vec{w}_d))$.

\begin{prop} \label{prop: divisor estimate} Let $v_1, \ldots, v_d \in \bZ^r$ be linearly independent vectors. Then for all $a_1, \ldots a_d \in \bZ$ and $q \in \bZ_{>0}$ with $\text{gcd}(a_1, \ldots, a_d, q) = 1$ we have that $$\text{gcd} (a_1v_1 + \cdots + a_dv_d, q) \leq d! \max_i  \| v_i\|^d_{\infty}.$$

\end{prop}

\begin{proof} Let $V_0:\bZ^d \to \bZ^r$ be the linear map given by $$V_0(x_1, \ldots, x_d) = \sum_{i=1}^d x_i v_i.$$ It is of full rank hence there exists a full rank $d \times d$ minor of the matrix $V_0$, in other words there is a projection $\pi: \bZ^r \to \bZ^d$ of co-ordinates so that $V = \pi \circ V_0: \bZ^d \to \bZ^d$ is of full rank. By the Smith-Normal-Form for integer matrices, there exists a linear map $D: \bZ^d \to \bZ^d$ and automorphisms $R$ and $L$ of $\bZ^d$ such that $$V = LDR$$ and $D$ is a diagonal matrix with non-zero (the kernel of $V$ and hence $D$ is trivial) diagonal entries satisfying the divisibility condition $D_{1,1} | D_{2,2} | \cdots | D_{d,d}$. Since automorphisms preserve divisors, we have that for $\vec{a} = (a_1, \ldots, a_d)$ with $\text{gcd}(a_1, \ldots, a_d, q) = 1$ that $\text{gcd}(R\vec{a}, q) = 1$. Hence since all $D_{i,i} \neq 0$, we get that $\text{gcd}(D R\vec{a}, q) \leq D_{d,d}.$ Since $L$ preserves divisors, we get that $\text{gcd}(V \vec{a}, q) = \text{gcd}(LDR\vec{a},q) \leq D_{d,d}$. We have the upper bound $$D_{d,d} \leq |\operatorname{det}(D)| = |\operatorname{det}(V)| \leq d! \max_i  \| v_i\|_{\infty}^d.$$ Finally, since $V = \pi \circ V_0$ we have $\operatorname{gcd}(V_0\vec{a}, q) \leq \operatorname{gcd}(V \vec{a}, q)$, which completes the proof.
\end{proof}

\begin{mydef} We say that a vector $P(x) = (P_1(x), \ldots, P_r(x))$, where $P_i(x) \in \bZ[x]$, has \textit{multiplicative complexity} $Q$ if for all $\vec{a} = (a_1, \ldots, a_r) \in \bZ^r$ and $q \in \bZ$ with $\text{gcd}(a_1, \ldots, a_r, q) = 1$ we have that the polynomial $$\sum_{j=1}^D b_j x^j = (P(x) - P(0)) \cdot \vec{a}$$ satisfies $\text{gcd}(b_1, \ldots, b_D, q) \leq Q$.

\end{mydef}

Throughout this paper, if $A(x) \in M_{d \times d}(\bZ[x])$ is a matrix with polynomial integer matrices, then we let $\| A(x) \|$ denote the largest absolute value of a coefficient appearing in $A(x)$.

\begin{corol}\label{corol: hyperplane fleeing implies bounded multiplicative complexity} Let $A(x) \in M_{d \times d}(\bZ[x])$ be a matrix with integer polynomial entries and $w \in \bZ^d \setminus \{0\}$ such that the entries of the row vector $w^t (A(x) - A(0))$ are elements of $\bZ[x]$ that are linearly independent over $\bZ$. Then $w^tA(x)$ has multiplicative complexity $Q$ where $$Q = Q(A(x), w) = d! \cdot \left( d \cdot \|A(x) - A(0)\| \|w \|_{\infty} \right)^d.$$

\end{corol}

\begin{proof} Let $v_1, \ldots, v_d \in \bZ[x]$ denote the entries of the row vector $w^t(A(x) - A(0))$. These are linearly independent over $\bZ$ and so we may apply Proposition~\ref{prop: divisor estimate} by viewing $v_i$ as an element of $\bZ^r$, where $r - 1$ is the maximal degree of the $v_i$, to obtain the desired estimate. \end{proof}

Throughout this paper, we let $e(t) = \exp(2\pi i t)$. We will need the following classical bound of Hua.

\begin{thm}[\cite{Hua}, see also \cite{HuaBook}] \label{thm: Hua bound} For $\delta > 0$ and positive integers $D$ there exists a constant $C_{D, \delta}$ such that if $f = a_0 + a_1 x + \cdots + a_D x^D \in \bZ[x]$ is a polynomial and $q$ is a positive integer such that $gcd(a_1, \ldots, a_D ,q) = 1$ then $$\left| \frac{1}{q} \sum_{n=1}^q e \left(\frac{f(n)}{q} \right) \right| \leq C_{D, \delta} q^{\delta - \frac{1}{D}}. $$
\end{thm}

We now state some extensions of tools developed by Alon-Peres \cite{Alon-Peres} that have been used or slightly modified in subsequent works on the Glasner property \cite{Dong1}, \cite{Kelly-Le}. Let $$B(M) = \{ \vec{m} \in \bZ^d ~|~ \vec{m} \neq \vec{0} \text{ and } \|\vec{m}\|_{\infty} \leq M \}$$ denote the $L^{\infty}$ ball of radius $M$ in $\bZ^d$ around $\vec{0}$ with $\vec{0}$ removed.

\begin{prop} \label{prop: fourier estimate for k^2} For each positive integer $d$ there exists a constant $C_1 = C_1(d) >0$ such that for all $\epsilon >0 $ if we set $M = \lfloor d/\epsilon \rfloor$ then the following is true: Let $\gamma_1, \ldots, \gamma_N \subset \text{M}_{d \times d} (\bZ)$ be a finite sequence of matrices and $X = \{x_1, \ldots, x_k\} \subset \bT^d$. Suppose that $\gamma_n X$ is not $\epsilon$-dense in $\bT^d$ for all $n=1,\ldots N$. Then 

$$k^2 \leq \frac{C_1}{\epsilon^d}\sum_{ \vec{m} \in B(M)} \sum_{1 \leq i,j \leq k} \frac{1}{N} \sum_{n=1}^N e(\vec{m} \cdot \gamma_n (x_i - x_j))  $$ \end{prop}

\begin{proof} This is exactly Proposition 2 in \cite{Kelly-Le} without the limit. See the short half-page proof that uses the exponential sum estimate from \cite{Barton-Montgomery-Hugh-Valeer}. \end{proof}

\begin{prop} \label{prop: h_q sum bound} Fix an integer $d>0$ and any real number $r>0$. Then there exists a constant $C = C(d,r)$ such that the following is true: Given any distinct $x_1, \ldots, x_k \in \bT^d$ let $h_q$ denote the number of pairs $(i,j)$ with $1 \leq i,j \leq k$ such that $q$ is the minimal (if such exists) positive integer such that $q(x_i - x_j) = 0$. Then $$\sum_{q=2}^{\infty} h_q q^{-r} \leq C k^{2 - r/(d+1)}.$$

\end{prop}

\begin{proof} For $r>1$, this is a combination of Proposition 5 and Lemma 4.2 in \cite{Dong1}, which is based on Proposition 1.3 of the Alon-Peres work \cite{Alon-Peres}. It is only stated in \cite{Dong1} for $r>1$ but it is in fact true for $r>0$. We reproduce the proof for the sake of convenience and certifying that indeed only the assumption $r>0$ is needed. Let $H_m = \sum_{q=2}^{m} h_q$ for $m \geq 2$ and $H_1 = 0$. We first show that $H_m \leq km^{d+1}$. To see this, note that for each fixed $i$ and $q$, there are at most $q^d$ values of $j$ such that $q(x_i - x_j) = 0$. Thus summing over $j = 1, \ldots, k$ and then over $q = 1, \ldots , m$ we get  $H_m \leq km^{d+1}$. Note also that $H_m \leq k^2$ for all $m$. Choose large enough $Q > k^{1/(d+1)}$ such that $h_q = 0$ for all $q > Q$. We have that \begin{align*} \sum_{q=2}^{\infty}h_q q^{-r} & = \sum_{q=2}^Q h_q q^{-r} \\ 
&= \sum_{q=2}^Q (H_q - H_{q-1}) q^{-r} \\
&= \sum_{q=2}^Q H_q (q^{-r} - (q+1)^{-r}) + H_Q(Q+1)^{-r} \\
&=\sum_{2 \leq q < k^{1/(d+1)}} H_q (q^{-r} - (q+1)^{-r}) + \sum_{k^{1/(d+1)} \leq q \leq Q} H_q (q^{-r} - (q+1)^{-r}) + H_Q(Q+1)^{-r} \end{align*}

Now for the second sum use the bound $H_q \leq k^2$, telescoping and let $Q \to \infty$. Then for the first sum use the inequality $H_q \leq kq^{d+1}$ to get

\begin{align*} \sum_{q=2}^{\infty}h_q q^{-r} &\leq  \sum_{2 \leq q < k^{1/(d+1)}} kq^{d+1}(q^{-r} - (q+1)^{-r}) + k^2 k^{-r/(d+1)} \\
&\leq k\sum_{2 \leq q < k^{1/(d+1)}} rq^{d - r} + k^2 k^{-r/(d+1)} \\
&\leq C k^{2 - r/(d+1)} +k^{2 - r/(d+1)}   \end{align*}
for some constant $C = C(d - r)$. \end{proof}

We are now ready to prove the (\ref{hyperplane fleeing main thm}) $\implies$ (\ref{glasner conditon main thm}) direction of Theorem~\ref{thm: main polynomial intro}. We will actually prove the following stronger quantitative form. 

\begin{thm}\label{thm: quantitative Glasner polynomial} For $\delta >0$ and integers $d, D>0$ there exists a constant $C_{\delta, d, D} >0$ such that the following is true: Let $A(x) \in M_{d \times d}(\bZ[x])$ be a matrix with integer polynomial entries of degree at most $D$ such that for each $w \in \bZ^d\setminus\{0\}$ the orbit $\{A(n)w ~|~ n \in \bZ\}$ is hyperplane-fleeing. Then for each $\epsilon>0$ and positive integers $$k > C_{\delta, d, D} \|A(x) - A(0)\|^{d(d+1)} \epsilon^{-2d(d+1)D - d(d+1) - \delta}$$ we have that whenever $x_1, \ldots, x_k$ are $k$ distinct elements of $\bT^d$ then there exists an integer $n$ such that $\{ A(n)x_1, \ldots, A(n)x_k \}$ is $\epsilon$-dense in $\bT^d$. 

\end{thm}

\begin{proof}

Fix $\epsilon >0$ and assume that no such $n$ exists. We will obtain an upper bound for $k$ by applying Proposition~\ref{prop: fourier estimate for k^2} with $\gamma_n = A(n)$ and letting $N \to \infty$ in the upper bound. We claim that if $x_i - x_j$ is irrational and $\vec{m} \in B(M)$ then $$  \lim_{N \to \infty} \frac{1}{N} \sum_{n=1}^N e(\vec{m} \cdot \gamma_n (x_i - x_j))  = \lim_{N \to \infty} \frac{1}{N}\sum_{n=1}^N  e(\vec{m}^t A(n)(x_i - x_j)) = 0.$$ To see this, first note that the row vector $$ \vec{m}^t (A(x) - A(0)) = [P_1(x), \ldots, P_d(x)] $$ has linearly independent entries over $\bZ$ (see Lemma~\ref{lemma: basic hyperplane-fleeing}) and hence over $\bR$ as $P_i(x) \in \bZ[x]$. Now if $\theta = (\theta_1, \ldots, \theta_d) \in \bT^d$ is irrational, then we claim that $$q(x) = \vec{m}^t (A(x) - A(0)) \theta = \sum \theta_i P_i(x)$$ is irrational, i.e., not in $\bQ[x]$. To see this, note that otherwise we have that $q(x), P_1(x), \ldots, P_d(x)$ are linearly dependent over $\bR$ and hence over $\bQ$ and so as $P_1(x), \ldots, P_d(x)$ are linearly independent we must have a linear combination $q(x) = \sum \theta'_i P_i(x)$ with all $\theta'_i \in \bQ$. But by linear independence of $P_1(x), \ldots, P_d(x)$ we have that $\theta_i = \theta'_i \in \bQ$. So we have shown that $\vec{m}^t A(n)(x_i - x_j)$ has at least one irrational non-constant coefficient when viewed as an element of $\bR[n]$ and hence by Weyl equidistribution we get the desired limit 

$$ \lim_{N \to \infty} \frac{1}{N}\sum_{n=1}^N  e(\vec{m}^t A(n)(x_i - x_j)) = 0.$$

Now we need to focus on the case where $x_i - x_j$ is rational. Thus we may write $x_i - x_j = \frac{1}{q}\vec{a}$ where $q \in \bZ_{>0}$ and $\vec{a} = (a_1, \ldots, a_d) \in \bZ^d$ with $\text{gcd}(q, a_1, \ldots, a_d) = 1$. Now by Corollary~\ref{corol: hyperplane fleeing implies bounded multiplicative complexity} we have that $\vec{m}^t A(x)$ has multiplicative complexity $Q$ where 
\begin{align} \label{Q bounded by eps^-d} Q = \sup_{\vec{m} \in B(M)} d! \cdot \left( d \cdot \|A(x) - A(0)\| \|\vec{m} \|_{\infty} \right)^d \leq d! d^{2d} \|A(x) - A(0)\|^d \epsilon^{-d} .\end{align} 

Thus the greatest common divisor of $q$ and the non-constant coefficients of the polynomial $\vec{m}^t A(x)\vec{a} \in \bZ[x]$ is at most $Q$. Thus if $D$ is the maximum degree of an entry in $A(x)$, we may apply Hua's bound (Theorem~\ref{thm: Hua bound}) to obtain a constant $C_2 = C_2(D, \delta)$ depending only on $D$ and any constant $0 < \delta < \frac{1}{D}$ such that

$$ \lim_{N \to \infty} \frac{1}{N} \sum_{n=1}^N e(\vec{m}^t A(n)(x_i - x_j)) = \frac{1}{q} \sum_{n=1}^q e\left(\frac{1}{q} \vec{m}^t A(n)\vec{a} \right) \leq C_{2}  \left(\frac{Q}{q} \right)^{\frac{1}{D} - \delta} $$

Now let $h_q$ denote the number of pairs $x_i, x_j$ such that $q$ is the least positive integer for which $q(x_i - x_j) = 0$. We apply Proposition~\ref{prop: fourier estimate for k^2} to obtain that

\begin{align*} k^2 & \leq \frac{C_1}{\epsilon^d}\sum_{ \vec{m} \in B(M)} \left( \sum_{q=2}^{\infty} h_q C_{2}  \left(\frac{Q}{q} \right)^{\frac{1}{D} - \delta} +k \right) \\
& \leq Q^{\frac{1}{D} - \delta}C_2 (2M)^d \frac{C_1}{\epsilon^d}\sum_{q=2}^{\infty} h_q q^{\delta -\frac{1}{D}} + \frac{C_1}{\epsilon^d}(2M)^d k \end{align*}

Now apply Proposition~\ref{prop: h_q sum bound} to get that $$\sum_{q=2}^{\infty} h_q q^{\delta - \frac{1}{D}} \leq C_3 k^{2 - (\frac{1}{D} - \delta)/(d+1)}$$ for some constant $C_3 = C_3(d,D)$ depending only on $d$ and $D$. Thus we have shown that 

$$ k^2 \leq Q^{\frac{1}{D} - \delta}C_2 (2M)^d \frac{C_1}{\epsilon^d}C_3 k^{2 - (\frac{1}{D} - \delta)/(d+1)} + \frac{C_1}{\epsilon^d}(2M)^d k $$

Now using $M = \lfloor d/\epsilon \rfloor$ and the upper bound (\ref{Q bounded by eps^-d}) on $Q$ we have that $$k \leq C_{\delta, d, D} \|A(x) - A(0)\|^{d(d+1)} \epsilon^{-2d(d+1)D - d(d+1) - \delta}  $$ for some constant $C_{\delta, d, D}$ depending only on $d,D$ and any $\delta>0$.  \end{proof} 

\section{Applications to groups generated by unipotent matrices}

\subsection{Balls in the Cayley graph of a linear group}

Let $\Gamma \subset \operatorname{SL}_d(\bZ)$ be a group generated by elements $S \subset \Gamma$ and suppose that the linear action $\Gamma \curvearrowright \bR^d$ is irreducible. We let $$S_r = \{s_1 \cdots s_m ~|~ 0 \leq m \leq r \text{ and } s_1, \ldots, s_r \in S\}$$ denote the elements of $\Gamma$ that can be written as a product of at most $r$ elements of $S$ (including $1 \in S_r$ as it is the empty product), i.e., the ball of radius $r$ in the Cayley graph with respect to $S$.

\begin{lemma} \label{lemma: Cayley ball full dimension} For each $v \in \bR^d \setminus \{ 0 \}$, we have that $\bR\text{-span}(S_{d-1}v) = \bR^d$.
\end{lemma}

\begin{proof} 
For integers $r \geq 0$ let $V_r = \bR\text{-span}(S_rv)$. Suppose $r \geq 0$ is such that $V_r \neq \bR^d$. Then by irreducibility of $\Gamma$ and $v \neq 0$ we must have that $V_r$ is not $\Gamma$-invariant and hence not $S$-invariant. Thus $SV_r \not\subset V_r$, and so $S_{r+1}v \not\subset V_r$, which means $\text{dim}V_{r+1} \geq \text{dim}V_r + 1$. That is, we have shown that the nested sequence of subspaces $V_0 \subset V_1 \subset V_2 \subset \ldots$ is strictly increasing in dimension until the dimension is $d$, with $V_0 = \bR v$ of dimension $1$, hence $V_{d-1} = \bR^d$ as required.
\end{proof}

\begin{lemma} \label{lemma: hyperplane fleeing Cayley ball} If $d>1$ and $v \in \bR^d \setminus \{ 0 \}$, then $S_{d}v$ is hyperplane fleeing.

\end{lemma}

\begin{proof} Suppose not, thus there exists a proper linear subspace $W \lneqq \bR^d$ and $a \in \bR^d$ such that $S_d v \subset W+a$. As $d>1$, there exists an $s \in S$ such that $s v - v \neq 0$ (as otherwise $\bR v$ would be a one-dimensional, hence proper, $\Gamma$ invariant subspace). Now apply Lemma~\ref{lemma: Cayley ball full dimension} to $sv - v \neq 0$ to get that $S_{d-1}(sv - v) \not \subset W$. But this contradicts $S_dv \subset W+a$ since $$S_{d-1}(sv - v) \subset S_dv - S_dv \subset W+a - (W+a) = W. $$\end{proof}

\subsection{Constructing polynomials via unipotents}

The following Proposition together with Theorem~\ref{thm: quantitative Glasner polynomial} completes the proof of Theorem~\ref{thm: main thm unipotent intro}. 

\begin{prop}\label{prop: unipotent implies polynomial} Suppose that $S \subset \operatorname{SL}_d(\bZ)$ where $d>1$ and each $s \in S$ is a unipotent element and suppose that the action of $\Gamma = \langle S \rangle $ on $\bR^d$ is irreducible. Then there exists a matrix with integer polynomial entries $A(x) \in M_{d \times d}(\bZ[x])$ such that $A(n) \in \Gamma$ for all $n \in \bZ$ and $\{A(n)w ~|~ n \in \bZ\}$ is hyperplane-fleeing for all $w \in \bR^d\setminus\{0\}$.

\end{prop}

\begin{proof}

Write $S = \{u_1, \ldots, u_m\}$ where each $u_i$ is a unipotent element and use cyclic notation so that $u_i = u_{i+jm}$ for all $i, j \in \bZ$. Note that for each fixed $i$ the matrix $u_i ^n$ has entries that are integer polynomials in $n$ hence $$Q_N(n_1, \ldots, n_N) = \prod_{i=1}^{N} u_i^{n_i} \in M_{d \times d}(\bZ[n_1, \ldots, n_N]) $$ is a matrix with multivariate integer polynomial entries in the variables $n_1, \ldots, n_N$. Now let $N = dm$ and use Lemma~\ref{lemma: hyperplane fleeing Cayley ball} to get that $\{ Q_N(n_1, \ldots n_N)w ~|~ n_1, \ldots n_N \in \bZ\}$ is hyperplane-fleeing for all $w \in \bR^d\setminus\{0\}$. In other words, for each fixed $w \in \bR^d\setminus\{0\}$ if we let $P_1, \ldots, P_d \in \bR[n_1, \ldots, n_N]$ be the polynomials such that $$Q(n_1, \ldots, n_N)w = (P_1(n_1, \ldots, n_N), \ldots, P_d(n_1, \ldots, n_N))$$ then $P_1, \ldots, P_d, 1$ are linearly independent over $\bR$. But there exists a large enough $R \in \bZ_{>0}$ (independent of $w$) such that the substitutions $n_i \mapsto n_i^{R^{i-1}}$ induce a map $\bZ[n_1, \ldots, n_N] \to \bZ[n]$ that is injective on the monomials appearing in $Q_N(n_1, \ldots, n_N)$. Thus $P_1, \ldots, P_d, 1$ remain linearly independent over $\bR$ after making this substitution, thus $\{ Q(n, n^R, \ldots, n^{R^{N-1}})w ~|~ n \in \bZ \}$ is also hyperplane fleeing. So the proof is complete with $A(x) = Q(x, x^R, \ldots, x^{R^{N-1}})$.
\end{proof}

\end{document}